\documentclass[final,a4paper]{siamltex}
\usepackage{amsmath,amssymb,amsfonts}
\usepackage{color}
\usepackage{verbatim}
\usepackage{eepic,epic,epsfig,epstopdf}
\usepackage[]{graphics}
\usepackage{graphicx,inputenc}
\usepackage{subfigure}
\usepackage{changepage}
\usepackage{amsmath}
\usepackage{threeparttable}
\usepackage{extarrows}
\usepackage{url,color}
\usepackage{graphicx,subfig}
\graphicspath{{./pics/}}
\usepackage{bm,color,url}
\usepackage{algorithm,algorithmic}
\usepackage{geometry}
\usepackage{setspace}
\usepackage[notref,notcite]{showkeys}
\newtheorem{Theorem}{Theorem}[section]
\newtheorem{Lemma}[theorem]{Lemma}
\newtheorem{Proposition}[theorem]{Proposition}
\newtheorem{Definition}[theorem]{Definition}

\newtheorem{Remark}{Remark}[section]

\def\relu{\mathrm{ReLU}}
\def\relu2{\mathrm{ReLU}^2}

\title{Convergence Rate   Analysis for Deep Ritz Method  \footnote{This work has been presented at
workshop on "Data Driven Scientific Computing"  held in Southern University of Science and Technology, China, on November 21-22, 2020.}}
\author{Chenguang Duan \thanks{School of Mathematics and Statistics, Wuhan University, Wuhan 430072, P.R. China. (cgduan.math@whu.edu.cn)}\quad\and
Yuling Jiao\thanks{School of Mathematics and Statistics, and
Hubei Key Laboratory of Computational Science, Wuhan University, Wuhan 430072, P.R. China. (yulingjiaomath@whu.edu.cn)}\quad\and
Yanming Lai \thanks{School of Mathematics and Statistics,  Wuhan University, Wuhan 430072, P.R. China. (laiyanming@whu.edu.cn)}\quad\and
 Xiliang Lu\thanks{School of Mathematics and Statistics, and
Hubei Key Laboratory of Computational Science, Wuhan University, Wuhan 430072, P.R. China. (xllv.math@whu.edu.cn)}\quad\and
Jerry Zhijian Yang \thanks{School of Mathematics and Statistics, and
Hubei Key Laboratory of Computational Science, Wuhan University, Wuhan 430072, P.R. China. (zjyang.math@whu.edu.cn)}
}

\begin{document}
%\begin{spacing}{1.2}
\maketitle

\begin{abstract}
Using deep neural networks to solve PDEs has attracted  a lot of attentions recently.
However, why the deep learning method works is falling  far behind its empirical success.
In this paper, we provide a rigorous numerical analysis on deep Ritz method  (DRM) \cite{wan11} for second order elliptic equations with Neumann boundary conditions.
We establish the first   nonasymptotic convergence rate in $H^1$ norm   for   DRM  using   deep networks with $\mathrm{ReLU}^2$ activation functions.
In addition to providing a theoretical justification of DRM, our study  also shed light   on how to set the  hyper-parameter  of depth and width to achieve the desired convergence rate in terms of number of training samples.
Technically, we derive   bound on  the approximation error of deep $\mathrm{ReLU}^2$ network in $H^1$ norm and  bound on the Rademacher complexity  of the non-Lipschitz composition of  gradient norm  and     $\mathrm{ReLU}^2$ network, both of which are of independent interest.
% We establish an nonasymptotic upper bound on the
%expected error  with  the number of samples  used in training.
\end{abstract}

\section{Introduction}
\par Partial differential equations (PDEs) have broad applications in physics, chemistry, biology, geology and engineering. A great deal of efforts have been devoted to  studying   numerical methods for solving PDEs \cite{brenner2007mathematical,ciarlet2002finite,leveque2007finite,wan14,wan15}.
 However, it is still a challenging task to develop  numerical scheme for  solving PDEs in  high-dimension.
  %Due to the curse of dimensionality, solving high-dimensional PDEs with traditional numerical methods becomes computationally intractable.
Due to the success of deep learning for high-dimensional data analysis  in computer  vision and  natural language processing, people have been paying  more attention to using (deep) neural network to solve PDEs in high dimension with may be  complex domain, an idea   that goes back to 1990's \cite{lee1990neural,lagaris1998artificial}.
%Deep neural network is an efficient tool in solving high-dimensional problems, which has been proven successful, for instance, in computer vision (CV) and natural language processing (NLP).
%This suggests that deep neural network can be applied to the arena of scientific computation to overcome the curse of dimensionality and deal with complex geometry domains.
In the last few years, there are  growing literatures  on neural network based numerical methods for PDEs %\cite{CiCP-28-2158,dgm4142,wan12,wan11,dgm02,dgm15}.
These works can be roughly  classified into two categories.

 In the first category, deep neural networks are used to improve classical methods.
\cite{artificial_viscosity} designs a neural network to estimate artificial viscosity in discontinuous Galerkin  schemes,  see also \cite{chen2020friedrichs}.
\cite{indicator} trains a neural network serving as a troubled-cell indicator in high-resolution schemes for conservation laws.
\cite{CiCP-28-2075} proposes a universal discontinuity detector using convolution neural network  and applies it in conjunction of solving nonlinear conservation.
\cite{CiCP-28-2158} uses reinforcement learning to find new and potentially better data-driven solvers for conservation laws.
%As such methods depend crucially on mesh-based discretization, methods in this category can not overcome the curse of dimensionality.

In the second category, deep neural networks are utilized to approximate the solution of the PDEs directly. Being  benefit from the excellent approximation power   of deep neural networks and SGD training, these methods have been successfully applied to    solve PDEs in high-dimension.
\cite{dgm02,dgm14} convert nonlinear parabolic PDEs into backward stochastic differential equations and solve them by deep neural networks, which can deal with high-dimensional problems.
Methods based on the strong form of PDEs \cite{raissi2019physics,wan12} are also proposed.
In \cite{raissi2019physics}, physics-informed neural networks (PINNs) use the squared residuals on the domain   as the loss function and treat  boundary conditions   as  penalty term.  There are several extensions of PINNs for different types of PDEs, including fractional PINNs \cite{fpinns}, nonlocal PINNs \cite{npinns}, conservative PINNs \cite{jagtap2020conservative}, eXtended PINNs \cite{xpinns}, among others.
A similar method presented in \cite{DeepXDE} proposes a residual-based adaptive refinement method to improve the training efficiency.
%There are some limitations to these approaches when the PDEs have singularities where classical solutions do not exist.

 In contrast to minimizing squared residuals of strong form, a natural alternative approach to derive loss functions are based on the variational form of PDEs \cite{wan11,wan}.
Inspired by Ritz method, \cite{wan11} proposes deep Ritz method (DRM) to solve  variational problems arising  from PDEs.
The idea of Galerkin method has also been used in \cite{wan}, where, they  propose a deep  Galerkin method  (DGM) via  reformulating   the problem of finding the weak solution of PDEs into an operator norm minimization problem induced by  the weak formulation.

\subsection{Related works and contributions}
Although there are great empirical achievements in recent
years as mentioned above, a challenging and interesting  question is that can we give
rigorous analysis to guarantee their  performances as people has done in
the classical counterpart  such as finite element method (FEM) \cite{ciarlet2002finite} and  finite difference method \cite{leveque2007finite} ?
Several recent efforts  have been devoted to making processes along this line.
\cite{luo2020two} consider the optimization and generalization error of second-order linear PDEs with two-layer neural networks in the scenario  of over-parametrization.
 \cite{shin2020convergence,mishra2020estimates,shin2020error} study the convergence of PINNs with deep neural networks. When we were about to finish our draft, we aware that \cite{lu2021priori} give an error analysis that focuses  on analyzing   one hidden layer  shallow networks   with ReLU-Cosine activation functions  to  solve elliptic PDEs whose  solutions are  restricted  to  spectral Barron space, see also  \cite{xu2020finite} for handling general equations with solutions living in   spectral Barron space via  two layer $\mathrm{ReLU}^k$ networks.
\textbf{ Two  important   questions have not been addressed   in the above mentioned related study are those:  what is the influence  of  the topological structure of the networks, say the   depth and width, in the quantitative  error analysis ?   How to determine these hyper-parameters to achieve a desired convergence rate ?}
In this paper, we give a firm answers on these  questions by  studying  convergence ratec of the deep Ritz method
to solve second order elliptic equations with Neumann boundary conditions
by using   $\mathrm{ReLU}^2$ networks with arbitrary depth.
As far as we know, we establish the first   nonasymptotic bound on   DRM.
The main contritions of this paper are summarized as follows.
\begin{itemize}
\item
 We derive  a bound on  the approximation error of deep $\mathrm{ReLU}^2$ network in $H^1$ norm, which is of independent interest,  i.e.,
 we prove that for any $u^*\in H^2(\Omega)$, there exist a $\mathrm{ReLU}^2$ network $\bar{u}_{\bar{\phi}}$ with depth $\mathcal{D} \leq \lceil\log_2d\rceil+3,$  width $ \mathcal{W}\leq \mathcal{O}(4d\left \lceil   \frac{1}{\epsilon}-4\right \rceil^d) $
such that $$\|u^*-\bar{u}_{\bar{\phi}}\|_{H^{1}\left(\Omega\right)}\leq \epsilon.$$
\item
We establish a  bound on  the statistical error in  DRM with the tools of Pseudo dimension, especially we give an bound on
$$\mathbb{E}_{Z_i,\sigma_i, i=1,...,n}[\sup_{u_{\phi}\in \mathcal{N}^2}\frac{1}{n}|\sum_{i} \sigma_i \|\nabla u_{\phi}(Z_i)\|^2|],$$
i.e., the Rademacher complexity  of the non-Lipschitz composition of  gradient norm  and   $\mathrm{ReLU}^2$ network,   via  calculating  the
Pseudo dimension of networks with both $\mathrm{ReLU}$ and  $\mathrm{ReLU}^2$ activation functions.  We believe  that the technique we used here is helpful for bounding the statistical errors for other deep PDEs solvers where the Rademacher  complexity of non-Lipschitz composition is hard to handle.
 \item Based on the above to error bounds we establish the first   nonasymptotic convergence rate of   deep Ritz method.
 We prove that
  if we set  the depth and width
 in $\mathrm{ReLU}^2$ networks to be  $$\mathcal{D} \leq \lceil\log_2d\rceil+3, \mathcal{W}\leq \mathcal{O}(4d\left \lceil   n^{\frac{1}{d+2+\nu}}-4\right \rceil^d),$$
 the $H^1$ norm  error of DRM in expectation is
 $${\mathcal{O}} (n^{-1/(2d+4+\nu)}),$$
 where $n$ is the number of training samples on  both the domain and the boundary, $\nu$ is a positive number that can be  an arbitrary small.
Our theory shed lights  on choosing the topological structure of the employed  networks to achieve the desired convergence rate in terms of number of training samples.
\item
By comparing the known results   in nonparametric regression, where the optimal   convergence rate   in $H^1$ norm for estimating functions in $H^2$ with $n$ paired samples  is  $\mathcal{O}(n^{-\frac{1}{4+d}})$ \cite{stone1982optimal},    we conjecture that the optimal convergence rate of DRM in $H^{1}$ norm is also ${\mathcal{O}} (n^{-1/(d+4)}).$
\end{itemize}

The rest of the paper are organized as follows.  In Section 2,  we give some preliminaries.  In Section 3, we  present the detail analysis
on the convergence rate of DRM.
We give conclusion and short discussion in Section 4.

\section{Preliminaries}
Consider the following elliptic equation with Neumann boundary conditions
\begin{equation}\label{mse}
\left\{\begin{aligned}
-\triangle u +  w u &= f  \text { in } \Omega \\
 \frac{\partial u}{\partial n} &= g   \text { on } \partial \Omega, \\
\end{aligned}\right.
\end{equation}
where, $\Omega$ is a bounded open  subset of $\mathbb{R}^{d}, d>1$,  $f(x)\in L^2(\Omega)$, $w(x) \in L^{\infty}(\Omega)$ satisfying $w(x) \geq c_1 > 0 $ a.e.,
and $g(s)\in  L^2(\partial \Omega)$. Without loss of generality we assume   $\Omega = (0,1)^{d}.$
 Define
\begin{equation}\label{lossp}
\mathcal{L}(u) = \frac{1}{2}|u|_{H^{1}(\Omega)}^{2} + \frac{1}{2}\|u\|_{L^{2}(\Omega;{w})}^{2}- \langle u, f\rangle_{L^2({\Omega})} - \langle {Tu}, g\rangle_{L^2({\partial \Omega})},
\end{equation}
where $T$ is the trace operator.
\begin{Lemma}\label{minimizer}
The unique weak solution   $u^* \in H^{1}(\Omega)$ of \eqref{mse} is the unique minimizer of $\mathcal{L}(u)$ over $H^{1}(\Omega)$.
Moreover, $u^* \in H^{2}(\Omega)$.
\end{Lemma}
\begin{proof}
Well known results, see for example \cite{evans1998partial}.
\end{proof}

A function $\mathbf{f}: \mathbb{R}^{d} \rightarrow \mathbb{R}^{N_{L}}$ implemented by a neural network is defined by
\begin{equation*}
\begin{array}{l}
\mathbf{f}_{0}(\mathbf{x})=\mathbf{x},\\
\mathbf{f}_{\ell}(\mathbf{x})=\varrho_{\ell}\left(A_{\ell} \mathbf{f}_{\ell-1}+\mathbf{b}_{\ell}\right) \quad \text { for } \ell=1, \ldots, L-1, \\
\mathbf{f}=\mathbf{f}_{L}(\mathbf{x}):=A_{L}\mathbf{f}_{L-1}+\mathbf{b}_{L},
\end{array}
\end{equation*}
where $A_{\ell}\in\mathbb{R}^{N_{\ell}\times N_{\ell-1}}$,
$\mathbf{b}_{\ell}\in\mathbb{R}^{N_{\ell}}$ and the activation function $\varrho_{\ell}$ is understood to act component-wise (note that here we allow different activation functions in different layers). $L$ is called the depth of the network and $\max\{N_{\ell},\ell=0,\cdots,L\}$ is called the width of the network. We will use $\mathcal{L}$ and $\mathcal{W}$ to denote the depth and width of  neural networks $\mathbf{f}$, respectively. $\sum_{\ell=1}^{L} N_{\ell} $ is called number of unites of $\mathbf{f}$ and  $\phi = \{A_{\ell},\mathbf{b}_{\ell}\}_{\ell}$ are called the weight parameters.  For simplicity we also use $\mathbf{f}_{\phi}$ to refer to the network.
We use $\mathcal{N}^{2}_{{\mathcal{D},\mathcal{W},\mathcal{B}}}$ to denote the set of neural networks  with  depth $\mathcal{D}$, width $\mathcal{W}$,  output  of the  function values and  their  square norm of gradients   bounded by $\mathcal{B}$,  activation function $\mathrm{ReLU}^2(x) = \max\{0,x^2\}$.
 Denote  $\mathcal{N}^{1,2}_{{\mathcal{D},\mathcal{W},\mathcal{B}}}$  as the set of neural networks  with  depth $\mathcal{D}$, width $\mathcal{W}$,  output bounded by $\mathcal{B}$,  activation functions $\mathrm{ReLU}(x) = \max\{0,x\}$ and  $\mathrm{ReLU}^2(x) = \max\{0,x^2\}$.
%$\varrho$ is called the activation function of the network. Note that the activation function in each neuron is not necessarily to be equal.

Obviously,
{ \begin{align*}
\mathcal{L}(u) &= |\Omega| \mathbb{E}_{{X\sim \mathrm{U}(\Omega)}}[\|\nabla u(X)\|_2^2/2 +w(X)u^2(X)/2- u(X)f(X)] \\
&-|\partial{\Omega}| \mathbb{E}_{{Y\sim \mathrm{U}(\partial\Omega)}}[{Tu}(Y)g(Y)],
\end{align*}}
where $\mathrm{U}(\Omega), \mathrm{U}(\partial\Omega)$ are the uniform distributions on $\Omega$ and $\partial \Omega$, respectively.
The main idea of  deep Ritz method (DRM) \cite{wan11} is employing
a  $u_{\phi} \in \mathcal{N}^{2}:= \mathcal{N}_{\mathcal{D},\mathcal{W},\mathcal{B}}^{2}$  to approximate the minimizer    $u^*$ of $\mathcal{L}$, i.e.,
 finding $u_{\phi}$ such that $\mathcal{L}(u_{\phi})$ closes to $ \mathcal{L}(u^*)$. To this end, by Lemma \ref{minimizer},
one may    consider the following empirical loss minimization problem
\begin{equation}\label{objsam}
\widehat{u}_{\phi}\in \min_{u_{\phi}\in \mathcal{N}^2} \widehat{\mathcal{L}}(u_{\phi}),
 \end{equation}
 where,
 \begin{align}\label{losss}
 \widehat{\mathcal{L}}(u_{\phi}) &= \frac{|\Omega|}{N}\sum_{i = 1}^N [\frac{\|\nabla u_{\phi}(X_i)\|_2^2}{2} +\frac{w(X_i)u_{\phi}^2(X_i)}{2}- u_{\phi}(X_i)f(X_i)] \nonumber \\
 & - \frac{\partial\Omega}{M} \sum_{j=1}^M [u_{\phi}(Y_j)g(Y_j)],
 \end{align}
 is a discrete version of the functional $\mathcal{L}(u_{\phi})$  with
  $\{X_i\}_{i = 1}^{N}$ being identically and independently distributed (i.i.d.) according to     $ \mathrm{U}(\Omega)$,  $\{Y_j\}_{j = 1}^{M} $ being identically and independently drawn  from  $ \mathrm{U}(\partial \Omega)$.
 % $\lambda = \frac{|\Omega|}{|\partial\Omega|} $.
{ Then, we call   a (random) solver   $\mathcal{A}$}, say  SGD,  to minimize (\ref{objsam}) and  denote the  output of $\mathcal{A}$, say $u_{\phi_{\mathcal{A}}} \in \mathcal{N}$,  as the final  solution.
\section{Error Analysis}
In this section we prove the convergence rate analysis for DRM with deep $\mathrm{ReLu}^2$ networks.
The following Lemma play an important role by decoupling the total errors into three types of errors.
\begin{Lemma}\label{errdec}
\begin{align*}
 &\|u_{\phi_{\mathcal{A}}}-u^{*}\|_{H^{1}(\Omega)}^2\\
 %&\mathcal{L}(u_{\phi_{\mathcal{A}}})-\mathcal{L}(u^{*})\\
 %= \\
% & \mathcal{L}(u_{\phi_{\mathcal{A}}}) - \widehat{\mathcal{L}}(u_{\phi_{\mathcal{A}}}) + \widehat{\mathcal{L}}(u_{\phi_{\mathcal{A}}}) -\widehat{\mathcal{L}}(\widehat{u}_{\phi}) +  \widehat{\mathcal{L}}(\widehat{u}_{\phi})   -   \widehat{\mathcal{L}}(\bar{u}_{\phi})\\
%  & +   \widehat{\mathcal{L}}(\bar{u}_{\phi}) - \mathcal{L}(\bar{u}_{\phi}) +\mathcal{L}(\bar{u}_{\phi})  - \mathcal{L}(u^*)\nonumber\\
  &\leq  \frac{2}{c_1\wedge 1} [\underbrace{\frac{\|w\|_{L^{\infty}(\Omega)}\vee 1}{2}\cdot \inf_{\bar{u}\in\mathcal{N}^2}\|\bar{u}-u^{*}\|_{H^1(\Omega)}^2}_{\mathcal{E}_{app}}  +\underbrace{2 \sup_{u \in \mathcal{N}^2} |\mathcal{L}(u) - \widehat{\mathcal{L}}(u) |}_{\mathcal{E}_{sta}}+ \underbrace{\widehat{\mathcal{L}}(u_{\phi_{\mathcal{A}}}) -\widehat{\mathcal{L}}(\widehat{u}_{\phi})}_{\mathcal{E}_{opt}}].
  \end{align*}
  \end{Lemma}
  \begin{proof}
For any $\bar{u}\in\mathcal{N}^2$, we have
\begin{equation*}
\begin{split}
&\mathcal{L}\left(u_{\phi_{\mathcal{A}}}\right)-\mathcal{L}\left(u^{*}\right) \\
&=\mathcal{L}\left(u_{\phi_{\mathcal{A}}}\right)-\widehat{\mathcal{L}}\left(u_{\phi_{\mathcal{A}}}\right) +\widehat{\mathcal{L}}\left(u_{\phi_{\mathcal{A}}}\right)-\widehat{\mathcal{L}}\left(\widehat{u}_{\phi}\right) +\widehat{\mathcal{L}}\left(\widehat{u}_{\phi}\right)-\widehat{\mathcal{L}}\left(\bar{u}\right) \\
&\quad+\widehat{\mathcal{L}}\left(\bar{u}\right)-\mathcal{L}\left(\bar{u}\right)+\mathcal{L}\left(\bar{u}\right)-\mathcal{L}\left(u^{*}\right) \\
&\leq \left[\mathcal{L}\left(\bar{u}\right)-\mathcal{L}\left(u^{*}\right)\right]+2 \sup _{u \in \mathcal{N}^2}\left|\mathcal{L}(u)-\widehat{\mathcal{L}}(u)\right| +\left[\widehat{\mathcal{L}}\left(u_{\phi_{\mathcal{A}}}\right)-\widehat{\mathcal{L}}\left(\widehat{u}_{\phi}\right)\right],
\end{split}
\end{equation*}
where the last step is due to the fact that $\widehat{\mathcal{L}}\left(\widehat{u}_{\phi}\right)-\widehat{\mathcal{L}}\left(\bar{u}\right)\leq 0$. Since $\bar{u}$ can be any element in $\mathcal{N}^2$, we take the infimum of $\bar{u}$ on both side of the above display,
\begin{align}
\mathcal{L}\left(u_{\phi_{\mathcal{A}}}\right)-\mathcal{L}\left(u^{*}\right)
&\leq \inf_{\bar{u}\in\mathcal{N}^2}\left[\mathcal{L}\left(\bar{u}\right)-\mathcal{L}\left(u^{*}\right)\right]+2 \sup _{u \in \mathcal{N}^2}\left|\mathcal{L}(u)-\widehat{\mathcal{L}}(u)\right| \nonumber\\ &+\left[\widehat{\mathcal{L}}\left(u_{\phi_{\mathcal{A}}}\right)-\widehat{\mathcal{L}}\left(\widehat{u}_{\phi}\right)\right].\label{errdec1}
\end{align}
Now for any $u\in\mathcal{N}^2$, set $v=u-u^*$, then
\begin{equation*}
\begin{split}
&\mathcal{L}\left(u\right)=\mathcal{L}\left(u^*+v\right)\\
&= \frac{1}{2}(\nabla (u^*+v),\nabla (u^*+v))_{L^{2}(\Omega)}+\frac{1}{2}(u^*+v,u^*+v)_{L^{2}(\Omega;w)}\\
&\quad\ -\langle u^*+v, f\rangle_{L^2({\Omega})} - \langle {Tu^*+Tv}, g\rangle_{L^2({\partial \Omega})}\\
&=\frac{1}{2}(\nabla u^*,\nabla u^*)_{L^{2}(\Omega)}+\frac{1}{2}(u^*,u^*)_{L^{2}(\Omega;w)}-\langle u^*, f\rangle_{L^2({\Omega})} - \langle {Tu^*}, g\rangle_{L^2({\partial \Omega})}\\
&\quad\ +\frac{1}{2}(\nabla v,\nabla v)_{L^{2}(\Omega)}+\frac{1}{2}(v,v)_{L^{2}(\Omega;w)}\\
&\quad\ +\left[(\nabla u^*,\nabla v)_{L^{2}(\Omega)}+(u^*,v)_{L^{2}(\Omega;w)}-\langle v, f\rangle_{L^2({\Omega})} - \langle {Tv}, g\rangle_{L^2({\partial \Omega})}\right]\\
&=\mathcal{L}\left(u^*\right)+\frac{1}{2}(\nabla v,\nabla v)_{L^{2}(\Omega)}+\frac{1}{2}(v,v)_{L^{2}(\Omega;w)},
\end{split}
\end{equation*}
where the last equality is due to the fact that $u^*$ is the weak solution of equation $(\ref{mse})$. Hence
\begin{align*}
\frac{c_1\wedge 1}{2}\|v\|_{H^1(\Omega)}^2\leq\mathcal{L}\left(u\right)-\mathcal{L}\left(u^*\right)
&=\frac{1}{2}(\nabla v,\nabla v)_{L^{2}(\Omega)}+\frac{1}{2}(v,v)_{L^{2}(\Omega;w)}\\
&\leq\frac{\|w\|_{L^{\infty}(\Omega)}\vee 1}{2}\|v\|_{H^1(\Omega)}^2,
\end{align*}
that is,
\begin{equation} \label{errdec2}
\frac{c_1\wedge 1}{2}\|u-u^*\|_{H^1(\Omega)}^2\leq\mathcal{L}\left(u\right)-\mathcal{L}\left(u^*\right)
\leq\frac{\|w\|_{L^{\infty}(\Omega)}\vee 1}{2}\|u-u^*\|_{H^1(\Omega)}^2.
\end{equation}
Combining $(\ref{errdec1})$ and $(\ref{errdec2})$ yields
\begin{align*}
&\|u_{\phi_{\mathcal{A}}}-u^{*}\|_{H^1(\Omega)}^2\\
&\leq \frac{2}{c_1\wedge 1}[\frac{\|w\|_{L^{\infty}(\Omega)}\vee 1}{2}\cdot \inf_{\bar{u}\in\mathcal{N}^2}\|\bar{u}-u^{*}\|_{H^1(\Omega)}^2\\
&+2 \sup _{u \in \mathcal{N}^2}\left|\mathcal{L}(u)-\widehat{\mathcal{L}}(u)\right| +\left[\widehat{\mathcal{L}}\left(u_{\phi_{\mathcal{A}}}\right)-\widehat{\mathcal{L}}\left(\widehat{u}_{\phi}\right)\right]].
\end{align*}
  \end{proof}

    The approximation  error $\mathcal{E}_{app}$  describes the expressive power of the  $\mathrm{ReLU}^2$ networks  $\mathcal{N}^2$ in $H^1$ norm, which  corresponds to the approximation error   in  FEM  known as the C$\acute{e}$a's lemma \cite{ciarlet2002finite}.
      The statistical error $\mathcal{E}_{sta}$  is caused by the monte carlo discritization of $\mathcal{L}(\cdot)$ defined in (\ref{lossp}) with $\widehat{\mathcal{L}}(\cdot)$ in (\ref{losss}).
    While, the optimization error $\mathcal{E}_{opt}$ indicates the perforce of the solver $\mathcal{A}$ we utilized. In contrast, this error is  is corresponding to the error of solving the linear systems in FEM.
    In this paper we focus on the first tow errors, i.e, considering the scenario of perfect training with $\mathcal{E}_{opt}=0$.
  \subsection{Approximation error}
The current literature on network  approximation theory  are mainly focus on the $L^p, p\in[1,+\infty]$ norm for  deep  networks \cite{yarotsky2017error,suzuki2018adaptivity,shen2019deep,montanelli2019new,schmidt2019deep,li2019better,siegel2020approximation}.
 %since the great empirical success of ReLU network in machine learning tasks.
  The  approximation error of ReLU network in Sobolev norm  are considered in \cite{guhring2020error, dung2020sparse}. However, the ReLU network  may not be
  suitable for solving PDEs since the term     $\nabla u_{\phi}$  in the loss function   will become   piece-wise constant  with respect to $\phi$,   which will prohibit using SGD for training.
  In this section  we derive an upper bound on the approximation error  $\mathrm{ReLU}^2$ networks  $\mathcal{N}^2$ in $H^1$ norm, which is of independent interest.
  \begin{Theorem}\label{errapp}
  Assume $\|u^*\|_{ H^{2}(\Omega)}\leq c_2$, then there exist an $\mathrm{ReLU}^2$ network $\bar{u}_{\bar{\phi}} \in \mathcal{N}^2$ with
%$$\mathcal{D} \leq\left\lceil 2 \log_{2} d\right\rceil+1,
%\mathcal{W} \leq \frac{C_{d}c_2^{d}}{\epsilon^{d}},$$
$$\mathcal{D} \leq \lceil\log_2d\rceil+3, \ \ \mathcal{W}\leq 4d\left \lceil   \frac{Cc_2}{\epsilon}-4\right \rceil^d $$
such that $$\mathcal{E}_{app} \leq \frac{\|w\|_{L^{\infty}(\Omega)}\vee 1}{2}\|u^*-\bar{u}_{\bar{\phi}}\|_{H^{1}\left(\Omega\right)}^2 \leq \frac{\|w\|_{L^{\infty}(\Omega)}\vee 1}{2}\epsilon^2,$$
where $C$ is a constant depending  only on  $d$.
  \end{Theorem}

\begin{proof}
Our proof is based on some classical  approximation results of  B-splines \cite{schumaker2007spline,de1978practical}.
Let us recall some notation and useful results.
%{\color{red} YM, please replace general k with $k =3$ in the following.}
We denote by $\pi_{l}$ the dyadic partition of $[0,1]$, i.e.,
\begin{equation*}
\pi_{l}: t_{0}^{(l)}=0<t_{1}^{(l)}<\cdots<t_{2^{l}-1}^{(l)}<t_{2^{l}}^{(l)}=1,
\end{equation*}
where $t_{i}^{(l)}=i \cdot 2^{-l}(0\leq i\leq 2^l)$. The cardinal B-spline of order $3$ with respect to partition $\pi_l$ is defined by
\begin{equation*}
N_{l, i}^{(3)}(x)=(-1)^k\left[t_{i}^{(l)}, \ldots, t_{i+3}^{(l)},(x-t)_{+}^{2}\right] \cdot\left(t_{i+3}^{(l)}-t_{i}^{(l)}\right),\quad i=-2,\cdots,2^l-1
\end{equation*}
which can be rewritten in the following equivalent form,
\begin{equation} \label{B-spline expression}
N_{l, i}^{(3)}(x)=2^{2l-1}\sum_{j=0}^{3}(-1)^{j}\left(\begin{array}{l}
3 \\
j
\end{array}\right)(x-i2^{-l}-j2^{-l})_{+}^{2},\quad i=-2,\cdots,2^l-1.
\end{equation}
The multivariate cardinal B-spline of order $3$ is defined by the product of univariate cardinal B-splines of order $3$, i.e.,
\begin{equation*}
{N}_{l, \mathbf{i}}^{(3)}(\mathbf{x})=\prod_{j=1}^{d} N_{l, i_j}^{(3)}\left(x_{j}\right), \quad \mathbf{i}=\left(i_{1}, \ldots, i_{d}\right),-3<i_{j}<2^{l}.
\end{equation*}
Denote
\begin{equation*}
S_{l}^{(3)}([0,1]^d)=\text{span}\{N_{l,\mathbf{i}}^{(3)},-3<i_j<2^l,j=1,2,\cdots,d\}.
\end{equation*}
Then, the element $f$ in $S_{l}^{(3)}([0,1]^d)$ are    piecewise  polynomial functions according to to partition $\pi_l^d$   with each piece being  degree $2$ and in $C^{1}([0,1]^d)$. Since
\begin{equation*}
S_1^{(3)}\subset S_2^{(3)}\subset S_3^{(3)}\subset\cdots,
\end{equation*}
We can further denote
\begin{equation*}
S^{(3)}([0,1]^d)=\bigcup_{l=1}^{\infty} S_{l}^{(3)}([0,1]^d).
\end{equation*}
 The following approximation result  of cardinal B-splines in Sobolev spaces which is a direct consequence of theorem 3.4 in  \cite{schultz1969approximation}  play an important role in the  proof of this Theorem.
\begin{Lemma}\label{spapp}
%Let $s\geq 1$, $k\geq s$ and $l\geq 1$.
 Assume  $u^*\in H^2([0,1]^d)$, there exists $\{c_j\}_{j=1}^{(2^l-4)^d}\subset\mathbb{R}$ with $l>2$ such that
\begin{equation*}
\|u^*-\sum_{j=1}^{(2^l-4)^d}c_j{N}_{l, \mathbf{i}_j}^{(3)}\|_{H^{1}(\Omega)}\leq \frac{C}{2^l}\|u^*\|_{H^1(\Omega)},
\end{equation*}
where $C$ is a constant only depend on $d$.
\end{Lemma}
%\begin{proof}
%If $k$ is even, we apply theorem 3.3 in $\cite{schultz1969approximation}$, otherwise we apply theorem 3.4 in $\cite{schultz1969approximation}$
%\end{proof}

\begin{Lemma}\label{spasrelu2}
The multivariate B-spline ${N}_{l, \mathbf{i}}^{(3)}(\mathbf{x})$ can be implemented exactly by a $\mathrm{ReLU}^2$  network with depth  $\lceil\log_2d\rceil+2$ and width $4d$.
\end{Lemma}
\begin{proof}
Denote
\begin{equation*}
\sigma (x)=\left\{\begin{array}{ll}
x^2, & x\geq 0 \\
0, & \text{else}
\end{array}\right.
\end{equation*}
as the activation function in $\mathrm{ReLU}^2$  network.
By definition of $N_{l, i}^{(3)}(x)$ in (\ref{B-spline expression}),  it's clear that $N_{l, i}^{(3)}(x)$ can be implemented by  $\mathrm{ReLU}^2$  network  without any error with depth  $2$ and width $4$. On the other hand  $\mathrm{ReLU}^2$  network  can also realize  multiplication without any error. In fact, for any $x,y\in\mathbb{R}$,
\begin{equation*}
xy = \frac{1}{4}[(x+y)^2-(x-y)^2]=\frac{1}{4}[\sigma (x+y)+\sigma (-x-y)-\sigma (x-y)-\sigma (y-x)].
\end{equation*}
Hence multivariate B-spline of order $3$  can be implemented by  $\mathrm{ReLU}^2$  network exactly with depth  $\lceil\log_2d\rceil+2$ and width $4d$.
\end{proof}

For any $\epsilon>0$, by Lemma \ref{spapp} and \ref{spasrelu2} with $\frac{1}{2^l} \leq \left \lceil \frac{C\|u^*\|_{H^2}}{\epsilon}\right \rceil$,
  there exists $\bar{u}_{\bar{\phi}} \in \mathcal{N}^2$, such that
\begin{equation} \label{H1 b-spline}
\left\|u^*-\bar{u}_{\bar{\phi}} \right\|_{H^1(\Omega)} \leq \epsilon.
\end{equation}
The depth $\mathcal{D}$ and width $\mathcal{W}$ of $\bar{u}_{\bar{\phi}}$ are satisfying   $\mathcal{D}\leq\lceil\log_2d\rceil+3$ and $\mathcal{W}\leq 4dn=4d\left \lceil   \frac{C\|u^*\|_{H^2}}{\epsilon}-4\right \rceil^d$, respectively.
\end{proof}

\subsection{Statistical error}
  In this section, we   bound the statistical error $$ {\mathcal{E}_{sta}}  = 2\sup_{u \in \mathcal{N}^2} |\mathcal{L}(u) - \widehat{\mathcal{L}}(u) |.$$
  \begin{Lemma}\label{errstadec}
  $$\sup_{u \in \mathcal{N}^2} |\mathcal{L}(u) - \widehat{\mathcal{L}}(u) | \leq \sum_{j=1}^4 \sup_{u \in \mathcal{N}^2} |\mathcal{L}_{j}(u) - \widehat{\mathcal{L}_{j}}(u)|,$$ where,
  $$\mathcal{L}_{1}(u) = |\Omega| \mathbb{E}_{{X\sim \mathrm{U}(\Omega)}}[w(X)u^2(X)/2], \ \  \widehat{\mathcal{L}_1}(u) = \frac{|\Omega|}{N}\sum_{i = 1}^N [\frac{w(X_i)u^2(X_i)}{2}],$$
  $$\mathcal{L}_{2}(u) = |\Omega| \mathbb{E}_{{X\sim \mathrm{U}(\Omega)}}[u(X)f(X)], \ \  \widehat{\mathcal{L}_2}(u) = \frac{|\Omega|}{N}\sum_{i = 1}^N [ u(X_i)f(X_i)],$$
  $$\mathcal{L}_{3}(u) = |\partial{\Omega}| \mathbb{E}_{{Y\sim \mathrm{U}(\partial\Omega)}}[{Tu}(Y)g(Y)], \ \  \widehat{\mathcal{L}_3}(u) = \frac{\partial\Omega}{M} \sum_{j=1}^M [u(Y_j)g(Y_j)],$$
   $$\mathcal{L}_{4}(u) = |\Omega| \mathbb{E}_{{X\sim \mathrm{U}(\Omega)}}[\|\nabla u(X)\|_2^2/2], \ \
   \widehat{\mathcal{L}_4}(u) = \frac{|\Omega|}{N}\sum_{i = 1}^N [\frac{\|\nabla u(X_i)\|_2^2}{2}].$$
  \end{Lemma}
  \begin{proof}
  This Lemma holds by the direct consequence of triangle inequality.
  \end{proof}

  By Lemma \ref{errstadec}, we have to bound the the maximum value of four random processes
  indexed by $u \in \mathcal{N}^2$. To this end, we recall tools in empirical process \cite{wellner,ledoux2013probability}.
   Denote
   \begin{equation}\label{boundb}
   \mathcal{B} = \max\{\|\bar{u}_{\bar{\phi}}\|_{L^{\infty}(\Omega)},\|\|\nabla \bar{u}_{\bar{\phi}}\|_2^2\|_{L^{\infty}(\Omega)}\},
   \end{equation}
   where $\bar{u}_{\bar{\phi}}$ is the best approximation of $u^*$ in Theorem \ref{errapp}.
  { Let  $f,g,w$ be  bounded, say by some constant  $c_3$, i.e, we  assume that  $$\|f\|_{L^{\infty}(\Omega)}\vee\|w\|_{L^{\infty}(\Omega)}\vee\|g\|_{L^{\infty}(\partial\Omega)} \vee \mathcal{B}  \leq c_3 < \infty.$$
%Let ${\color{red}\|\bar{\phi}\|_{\mathcal{D}}}  = \max \{\|\bar{\phi}_1\|,...\|\bar{\phi}_{\mathcal{D}}\|\}< c_5$ with $\bar{\phi}$ be the parameter of   $\bar{u}_{\bar{\phi}}$ in Theorem \ref{errapp}. We choose $\mathcal{N} = \{\mathrm{ReLu2 \ \ net} \ \ u_{\phi}:  \|\phi\|_{\mathcal{D}} < c_5\}$.
We use $\mu$ to denote  $ \mathrm{U}(\Omega) (\mathrm{U}(\partial \Omega))$.
 Given $n = N(M)$  i.i.d samples $\mathbf{Z}_n = \{Z_i\}_{i=1}^n$  from $\mu$, with $Z_i = X_i (Y_i) \sim \mu$, we need the following Rademacher complexity to measure the capacity of the given function class $\mathcal{N}$ restricted on $n$ random samples $\mathbf{Z}_n$.
 \begin{Definition}
 The Rademacher complexity of a set $A \subseteq \mathrm{R}^n$ is defined as
  $$\mathfrak{R}(A) = \mathbb{E}_{\mathbf{Z}_n,\Sigma_n}[\sup_{a\in A}\frac{1}{n}|\sum_{i} \sigma_i a_i|],$$
  where,   $\Sigma_n = \{\sigma_i\}_{i=1}^n$ are $n$ i.i.d  Rademacher variables with $\mathbb{P}(\sigma_i = 1) = \mathbb{P}(\sigma_i = -1) = \frac{1}{2}.$
  The Rademacher complexity of  function class $\mathcal{N}$ associate with random sample $\mathbf{Z}_n$ is defined as as $$\mathfrak{R}(\mathcal{N}) = \mathbb{E}_{\mathbf{Z}_n,\Sigma_n}[\sup_{u\in \mathcal{N}}\frac{1}{n}|\sum_{i} \sigma_i u(Z_i)|].$$
  \end{Definition}
\begin{Lemma}\label{lembound}
%$\|u_{\phi}\|_{L^{\infty}(\overline{\Omega})} \leq \Pi_{j=1}^{\mathcal{D}}\|\bar{\phi_j}\| R = {\color{red}c_6},$ $\forall u_{\phi} \in \mathcal{N}$.
 Let $\Psi_1(x,y) = \frac{w(x)y^2}{2}: \mathbb{R}^d\times \mathbb{R}, |y|\leq c_3$, $\Psi_2(x,y) = f(x)y : \mathbb{R}^d\times \mathbb{R}, |y|\leq c_3$, $  \Psi_3(x,y) = g(x)y: \mathbb{R}^d\times \mathbb{R}, |y|\leq c_3 $.  Then $\Psi_1(x,y)$, $\Psi_2(x,y)$  and $\Psi_3(x,y)$ are  $c_3^2$, $c_3$ and $c_3$-Lipschitz continuous on $y$  for all $x$ and $\Psi_j(x,0) =0, j=1,2,3$.
\end{Lemma}
\begin{proof}
We give  the proof for $\Psi_1$ and omit the details for $\Psi_2, \Psi_3$ since they can be shown similarly.
For arbitrary $y_1, y_2 $ with $|y_i|\leq c_3, i=1,2$,
\begin{align*}
|\Psi_1(x, y_1) - \Psi_1(x, y_2)| &= |\frac{w(x)y_{1}^2}{2} - \frac{w(x)y_{2}^2}{2}|\\
& = \frac{|w(x)(y_1+y_2)|}{2} |y_1-y_2| \leq c_3^2 |y_1(x)-y_2|.
\end{align*}
\end{proof}
%\item Given $n$  i.i.d samples $\mathbf{Z}_n = \{Z_i\}_{i=1}^n$  from $\mu$, with $Z_i = X_i (Y_i) \sim \mu = \mathrm{Uniform}(\Omega) ( \mathrm{Uniform}(\partial \Omega)), n = N (M)$.  Let  $\Sigma_n = \{\sigma_i\}_{i=1}^n$ be i.i.d  Rademacher variables, define the Rademacher complexity of $\mathcal{N}$ $$\mathfrak{R}(\mathcal{N}) = \mathbb{E}_{\mathbf{Z}_n,\Sigma_n}[\sup_{u\in \mathcal{N}}\frac{1}{n}|\sum_{i} \sigma_i u(Z_i)|].$$

 By Corollary 3.17 in \cite{ledoux2013probability} and Lemma \ref{lembound},    we have the following
 Lipschitz  contraction results on Rademacher complexity.
\begin{Lemma}\label{Redm1}
Let $\mathcal{N} = \{u(x): \|u\|_{L^{\infty}(\Omega)} \leq c_3\}$.
Define,
 $$\Psi_j \circ \mathcal{N}= \{\mathrm{composition \ \ of }  \ \  \Psi_j \ \  \mathrm{and} \ \ \mathcal{N}: x \rightarrow \Psi_j (x,u(x)): u\in \mathcal{N} \}, j =1,2,3.$$
 Then,
    $\mathfrak{R}(\Psi_1 \circ \mathcal{N}) \leq 2 c_3^2\mathfrak{R}(\mathcal{N}),$
     $\mathfrak{R}(\Psi_i \circ \mathcal{N}) \leq 2 c_3\mathfrak{R}(\mathcal{N}), i = 2,3$.
%Similarly, Let $ \Psi(x,y) = -f(x)y,  g(x)y$, we get  $\mathfrak{R}(\Psi \circ \mathcal{N}) \leq 2 c_4\mathfrak{R}(\mathcal{N}).$
\end{Lemma}

The following symmetrization result shows that the Rademacher complexity   $\mathfrak{R}(\Psi_j \circ \mathcal{N}^2)$ gives upper bound on
$\sup_{u \in \mathcal{N}^2} |\mathcal{L}_{j}(u) - \widehat{\mathcal{L}_{j}}(u)|,$ $j =1,...,3.$
\begin{Lemma}\label{sym}
$\mathbb{E}_{\mathbf{Z}_n}[\sup_{u \in \mathcal{N}^2} |\mathcal{L}_{j}(u) - \widehat{\mathcal{L}_{j}}(u)|]\leq \mathfrak{R}(\Psi_j \circ \mathcal{N}^2)$, $j =1,...,3.$
\end{Lemma}
\begin{proof}
We give the proof for  $j=1$ and omit the proof for $\Psi_2$ and $\Psi_3$ since they can be shown similarly.
Let $\tilde{\mathbf{Z}}_n = \{\tilde{Z}_i\}_{i=1}^n$ be an i.i.d ghost sample  from $\mu$ and $\tilde{\mathbf{Z}}_n$ is independent of ${\mathbf{Z}}_n$.
 \begin{align*}
 &\mathbb{E}_{\mathbf{Z}_n}[\sup_{u \in \mathcal{N}^2} |\mathcal{L}_{1}(u) - \widehat{\mathcal{L}_{1}}(u)|]\\
 %&=\mathbb{E}_{\mathbf{Z}_n}[\sup_{u \in \mathcal{N}^2} \mathcal{L}_{1}(u) - \widehat{\mathcal{L}_{1}}(u)]\\
  %&=\mathbb{E}_{\mathbf{Z}_n} [\sup_{u \in \mathcal{N}^2} |\mathcal{L}(D) - \widehat{\mathcal{L}}(D) | ] \\
  &=\frac{|\Omega|}{n}\mathbb{E}_{\mathbf{Z}_n} [\sup_{u \in \mathcal{N}^2} |\sum_{i=1}^n (\mathbb{E}_{\tilde{\mathbf{Z}}_n}[\frac{w(\tilde{Z}_i)u^2(\tilde{Z}_i)}{2}] -\frac{w(Z_i)u^2(Z_i)}{2})|]\\
  & \leq  \frac{|\Omega|}{n} \mathbb{E}_{\mathbf{Z}_n, \tilde{\mathbf{Z}}_n} [\sup_{u \in \mathcal{N}^2} |\sum_{i =1}^n(\frac{w(\tilde{Z}_i)u^2(\tilde{Z}_i)}{2} -\frac{w(Z_i)u^2(Z_i)}{2})|]\\
  & =   \frac{|\Omega|}{n} \mathbb{E}_{\mathbf{Z}_n, \tilde{\mathbf{Z}}_n,\Sigma_n} [\sup_{u \in \mathcal{N}^2} |\sum_{i=1}^n\sigma_i (\frac{w(\tilde{Z}_i)u^2(\tilde{Z}_i)}{2} -\frac{w(Z_i)u^2(Z_i)}{2})|]\\
  & \leq \frac{|\Omega|}{n}  \mathbb{E}_{\mathbf{Z}_n, \Sigma} [\sup_{u\in \mathcal{N}^2} |\sum_{i=1}^n \frac{w({Z}_i)u^2({Z}_i)}{2}| ] + \frac{|\Omega|}{n}  \mathbb{E}_{\tilde{\mathbf{Z}}_n,\Sigma} [\sup_{u\in \mathcal{N}^2} |\sum_{i=1}^n \frac{w(\tilde{Z}_i)u^2(\tilde{Z}_i)}{2}| ] \\
  &= 2|\Omega|\mathfrak{R}(\Psi_1 \circ \mathcal{N}^2)
  \end{align*}
  where, %the first equality due to $\mathcal{N}^2$ is closed under negation,
  the first inequality follows from the Jensen's inequality, and the second equality holds since  both $\sigma_i(\frac{w(\tilde{Z}_i)u^2(\tilde{Z}_i)}{2} - \frac{w({Z}_i)u^2({Z}_i)}{2})$ and $(\frac{w(\tilde{Z}_i)u^2(\tilde{Z}_i)}{2} - \frac{w({Z}_i)u^2({Z}_i)}{2})$ are governed by the same law, and the last  equality holds since the distribution of the two terms are the same.
\end{proof}

Next we give a upper bound of $\mathfrak{R}(\mathcal{N}^2)$ in terms of the covering number   of $\mathcal{N}^2$  by using  the   Dudley's entropy formula \cite{dudley}.
\begin{Definition}
Suppose that $W\subset\mathbb{R}^n$. For any $\epsilon>0$, let $V\subset\mathbb{R}^n$ be a $\epsilon$-cover of $W$ with respect to the distance  $d_{\infty}$, that is, for any $w\in W$, there exists a $v\in V$ such that $d_{\infty}(u,v)<\epsilon$, where $d_{\infty}$ is defined by
\begin{equation*}
d_{\infty}(u,v):=\|u-v\|_{{\infty}}.
\end{equation*}
The covering number $\mathcal{C}(\epsilon,W,d_{\infty})$ is defined to be the minimum cardinality among all $\epsilon$-cover of $W$ with respect to the distance  $d_{\infty}$.
\end{Definition}

\begin{Definition}
Suppose that $\mathcal{N}$ is a class of functions from $\Omega$ to $\mathbb{R}$. Given $n$  sample  $\mathbf{Z}_n=(Z_1,Z_2,\cdots,Z_n) \in \Omega^n$, $\mathcal{N}|_{\mathbf{Z}_n}\subset\mathbb{R}^n$ is defined by
\begin{equation*}
\mathcal{N}|_{\mathbf{Z}_n}= \{(u(Z_1),u(Z_2),\cdots, u(Z_n)): u\in\mathcal{N}\}.
\end{equation*}
The uniform covering number $\mathcal{C}_{\infty}(\epsilon,\mathcal{N},n)$ is defined by
\begin{equation*}
\mathcal{C}_{\infty}(\epsilon,\mathcal{N},n)=\max_{\mathbf{Z}_n\in\Omega^n}\mathcal{C}(\epsilon,\mathcal{N}|_{\mathbf{Z}_n},d_{\infty})
\end{equation*}
\end{Definition}

\begin{Lemma}\label{Redm2}
Assume $0\in \mathcal{N}$ and the diameter of $\mathcal{N}$ is less than $\mathcal{B}$, i.e., $\|u\|_{L^\infty(\Omega)} \leq \mathcal{B}, \forall u \in \mathcal{N}$.  Then
$$\mathfrak{R}(\mathcal{N}) \leq  \inf_{0<\delta<\mathcal{B}}\left(4 \delta+\frac{12}{\sqrt{n}} \int_{\delta}^{\mathcal{B}} \sqrt{\log (2\mathcal{C}\left( \epsilon, \mathcal{N}, n\right))}\mathrm{d}\epsilon\right).$$
\end{Lemma}
\begin{proof}
The proof follows from the chaining method \cite{wellner}.
We need the  Massart's finite class Lemma \cite{boucheron2013concentration}.
\begin{Lemma}\label{Mfinite}
For any finite set $V \in \mathbb{R}^n$ with diameter $D = \sup_{v\in V}\|v\|_2$, then
$$\mathbb{E}_{\Sigma_n} [\sup_{v\in V} \frac{1}{n}|\sum_i \sigma_i v_i| ]\leq \frac{D}{n}\sqrt{2\log (2|V|)}.$$
\end{Lemma}
By definition $$\mathfrak{R}(\mathcal{N}) = \mathfrak{R}(\mathcal{N}|_{\mathbf{Z}_n}) = \mathbb{E}_{\mathbb{Z}_n} [\mathbb{E}_{\Sigma}[\sup_{u\in \mathcal{N}}\frac{1}{n}|\sum_{i} \sigma_i u(Z_i)| |\mathbf{Z}_n]].$$
Thus, it suffice to show $$\mathbb{E}_{\Sigma}[\sup_{u\in \mathcal{N}}\frac{1}{n}|\sum_{i} \sigma_i u(Z_i)|]
\leq  \inf_{0<\delta< \mathcal{B}}\left(4 \delta+\frac{12}{\sqrt{n}} \int_{\delta}^{\mathcal{B}} \sqrt{\log \mathcal{C}\left( \epsilon, \mathcal{N}^2, n\right)}\mathrm{d}\epsilon\right)$$ by conditioning  on $\mathbf{Z}_n$.
Given an positive integer $K$, let $\epsilon_k = 2^{-k+1}\mathcal{B}$, $k=1,...K$. Let $C_k$ be a cover of $\mathcal{N}|_{\mathbf{Z}_n}\subseteq\mathbb{R}^n$ whose   covering number is denoted as  $\mathcal{C}(\epsilon_k,\mathcal{N}|_{\mathbf{Z}_n},d_{\infty})$. Then, by definition,
$\forall u \in \mathcal{N},$ there $\exists$ $c^k \in C_k$ such that $$d_{\infty}(u|_{\mathbf{Z}_n},c^k) =\max\{|u(Z_i)-c^k_i|, i =1,...,n\}\leq \epsilon_k, k =1,...,K.$$ Moreover, we denote
the best approximate element of $u$ in $C_k$ with respect to $d_{\infty}$ as $c^k(u)$.
Then, \begin{align*}
&\mathbb{E}_{\Sigma}[\sup_{u\in \mathcal{N}}\frac{1}{n}|\sum_{i=1}^n \sigma_i u(Z_i)|]\\
&= \mathbb{E}_{\Sigma}[\sup_{u\in \mathcal{N}}\frac{1}{n}|\sum_{i=1}^n \sigma_i (u(Z_i)-c^K_i(u)) +\sum_{j=1}^{K-1}\sum_{i=1}^n \sigma_i(c^j_i(u)-c^{j+1}_i(u))+\sum_{i=1}^n \sigma_i c^{1}_{i}(u)|]\\
& \leq  \mathbb{E}_{\Sigma}[\sup_{u\in \mathcal{N}}\frac{1}{n}|\sum_{i=1}^n \sigma_i (u(Z_i)-c^K_i(u))|]
 +\sum_{j=1}^{K-1} \mathbb{E}_{\Sigma}[\sup_{u\in \mathcal{N}}\frac{1}{n}|\sum_{i=1}^n \sigma_i(c^j_i(u)-c^{j+1}_i(u))|]\\
 &+\mathbb{E}_{\Sigma}[\sup_{u\in \mathcal{N}}\frac{1}{n}|\sum_{i=1}^n \sigma_i c^{1}_{i}(u)|].
\end{align*}
Since $0\in \mathcal{N}$, and the diameter of $\mathcal{N}$ is smaller than $\mathcal{B}$, we can choose $C_1 = \{0\}$ such that
the third term in the above display vanishes.
By  H\"{o}lder's inequality, we deduce that the first term can be bounded by $\epsilon_K$ as follows.
\begin{align*}
&\mathbb{E}_{\Sigma}[\sup_{u\in \mathcal{N}}\frac{1}{n}|\sum_{i=1}^n \sigma_i (u(Z_i)-c^K_i(u))|]\\
&\leq \mathbb{E}_{\Sigma}[\sup_{u\in \mathcal{N}}\frac{1}{n}(\sum_{i=1}^n|\sigma_i|)(\sum_{i=1}^n\max_{i=1,...,n}\{|u(Z_i)-c^K_i(u)|\})]\\
&\leq \epsilon_K.
\end{align*}
Let $V_j = \{c^j(u)-c^{j+1}(u): u\in \mathcal{N}\}$. Then by definition, the number of elements in $V_j$ and $C_j$ satisfying
$$|V_j|\leq |C_j||C_{j+1}|\leq |C_{j+1}|^2.$$ And the  diameter of $V_j$ denoted as $D_j$ can be bounded as
\begin{align*}
&D_j = \sup_{v\in V_j}\|v\|_2  \leq \sqrt{n}\sup_{u\in \mathcal{N}} \|c^j(u)-c^{j+1}(u)\|_{\infty} \\
& \leq \sqrt{n}\sup_{u\in \mathcal{N}} \|c^j(u)-u\|_{\infty} + \|u-c^{j+1}(u)\|_{\infty}\\
&\leq \sqrt{n}(\epsilon_j + \epsilon_{j+1})\\
 &\leq 3\sqrt{n} \epsilon_{j+1}.
\end{align*}
Then,
\begin{align*}
&\mathbb{E}_{\Sigma}[\sup_{u\in \mathcal{N}}\frac{1}{n}|\sum_{j=1}^{K-1}\sum_{i=1}^n \sigma_i(c^j_i(u)-c^{j+1}_i(u))|]\leq \sum_{j=1}^{K-1} \mathbb{E}_{\Sigma}[\sup_{v\in V_j}\frac{1}{n}|\sum_{i=1}^n \sigma_iv_j|]\\
&\leq \sum_{j=1}^{K-1} \frac{D_j}{n}\sqrt{2\log (2|V_j|)}\\
&\leq \sum_{j=1}^{K-1}  \frac{6\epsilon_{j+1}}{\sqrt{n}}\sqrt{\log (2|C_{j+1}|)},
\end{align*}
where we use triangle inequality in the  first inequality, and use Lemma \ref{Mfinite} in the second
inequality.
Putting all the above estimates together, we get
\begin{align*}
&\mathbb{E}_{\Sigma}[\sup_{u\in \mathcal{N}}\frac{1}{n}|\sum_{i} \sigma_i u(Z_i)|]\leq  \epsilon_K + \sum_{j=1}^{K-1}  \frac{6\epsilon_{j+1}}{\sqrt{n}}\sqrt{\log (2|C_{j+1}|)}\\
& \leq \epsilon_K + \sum_{j=1}^{K}  \frac{12(\epsilon_j-\epsilon_{j+1})}{\sqrt{n}}\sqrt{\log (2 \mathcal{C}\left( \epsilon_j, \mathcal{N}, n\right))}\\
&\leq  \epsilon_K +   \frac{12}{\sqrt{n}} \int_{\epsilon_{K+1}}^{\mathcal{B}}\sqrt{\log (2 \mathcal{C}\left( \epsilon, \mathcal{N}, n\right))}\mathrm{d}\epsilon\\
&\leq \inf_{0<\delta<\mathcal{B}}(4 \delta+\frac{12}{\sqrt{n}} \int_{\delta}^{\mathcal{B}} \sqrt{\log (2\mathcal{C}\left( \epsilon, \mathcal{N}, n\right))}\mathrm{d}\epsilon).
\end{align*}
where, last inequality holds since for $0<\delta< \mathcal{B}$, we can choose $K$ to be the largest integer such that  $\epsilon_{K+1} >\delta$, at this time $\epsilon_K\leq  4 \epsilon_{K+2} \leq 4 \delta.$
\end{proof}

Now we turn to handle the most difficult term  $\sup_{u \in \mathcal{N}^2} |\mathcal{L}_{4}(u) - \widehat{\mathcal{L}_{4}}(u)|$, where
we need to bound the Rademacher complexity  of the non-Lipschitz composition of  gradient norm  and   $\mathrm{ReLU}^2$ network.
 We believe that the technique we used here
is helpful for bounding the statistical errors for other deep PDEs solvers where the main difficulties is bounding the
Rademacher complexity of non-Lipschitz composition  induced by the gradient operator.
\begin{Lemma}\label{sym2}
\begin{align}
&\mathbb{E}_{\mathbf{Z}_n}[\sup_{u \in \mathcal{N}^2} |\mathcal{L}_{4}(u) - \widehat{\mathcal{L}_{4}}(u)|] \nonumber\\
&\leq \mathbb{E}_{\mathbf{Z}_n,\Sigma_n}[\sup_{u\in \mathcal{N}^2}\frac{1}{n}|\sum_{i} \sigma_i \|\nabla u(Z_i)\|^2|] \label{symg}\\
&\leq \mathfrak{R}(\mathcal{N}^{1,2}) = \mathbb{E}_{\mathbf{Z}_n,\Sigma_n}[\sup_{u\in \mathcal{N}^{1,2}}\frac{1}{n}|\sum_{i} \sigma_i u(Z_i)|]\label{upr12}\\
&  \leq  \inf_{0<\delta<\mathcal{B}}\left(4 \delta+\frac{12}{\sqrt{n}} \int_{\delta}^{\mathcal{B}} \sqrt{\log (2\mathcal{C}\left( \epsilon, \mathcal{N}^{1,2}, n\right))}\mathrm{d}\epsilon\right)\label{rg2c}.
\end{align}
\end{Lemma}
\begin{proof}
The proof of (\ref{symg}) is based on the symmetrization method used in the proof of Lemma \ref{sym}, we omit the detail here.
The proof of (\ref{upr12}) is a direct consequence of the following claim.
{Claim: Let $u$ be a function implemented by a $\relu2$ network with depth $\mathcal{D}$ and width $\mathcal{W}$. Then $\|\nabla u\|_2^2$ can be implemented by a $\mathrm{ReLU}$-$\relu2$ network with depth $\mathcal{D}+3$ and width $d\left(\mathcal{D}+2\right)\mathcal{W}$.}

Denote $\mathrm{ReLU}$ and $\relu2$ as $\sigma_1$ and $\sigma_2$, respectively.
As long as we show that each partial derivative  $D_iu(i=1,2,\cdots,d)$ can be implemented by a $\mathrm{ReLU}$-$\relu2$ network respectively, we can easily obtain the network we desire, since, $\|\nabla u\|_2^2=\sum_{i=1}^{d}\left|D_i u\right|^2$  and the square function can be implemented by  $x^2=\sigma_2(x)+\sigma_2(-x)$.

Now we show that for any $i=1,2,\cdots,d$, $D_iu$ can be implemented by a $\mathrm{ReLU}$-$\relu2$ network. We deal with the first two layers in details since there  are  a  little bit difference for the first two layer and apply induction for layers $k\geq3$. For the first layer, since $\sigma_2^{'}(x)=2\sigma_1(x)$, we have for any $q=1,2\cdots,n_1$
\begin{equation*}
D_iu_q^{(1)}=D_i\sigma_2\left(\sum_{j=1}^{d}a_{qj}^{(1)}x_j+b_q^{(1)}\right)
=2\sigma_1\left(\sum_{j=1}^{d}a_{qj}^{(1)}x_j+b_q^{(1)}\right)\cdot a_{qi}^{(1)}
\end{equation*}
Hence $D_iu_q^{(1)}$ can be implemented by a $\mathrm{ReLU}$-$\relu2$ network with depth $2$ and width $1$. For the second layer,
\begin{equation*}
D_iu_q^{(2)}=D_i\sigma_2\left(\sum_{j=1}^{n_1}a_{qj}^{(2)}u_j^{(1)}+b_q^{(2)}\right)
=2\sigma_1\left(\sum_{j=1}^{n_1}a_{qj}^{(2)}u_j^{(1)}+b_q^{(2)}\right)\cdot\sum_{j=1}^{n_1}a_{qj}^{(2)}D_iu_j^{(1)}
\end{equation*}
Since $\sigma_1\left(\sum_{j=1}^{n_1}a_{qj}^{(2)}u_j^{(1)}+b_q^{(2)}\right)$ and $\sum_{j=1}^{n_1}a_{qj}^{(2)}D_iu_j^{(1)}$ can be implemented by two $\mathrm{ReLU}$-$\relu2$ \verb"subnetworks", respectively, and the multiplication can also be implemented by
\begin{equation*}
\begin{split}
x\cdot y&=\frac{1}{4}\left[(x+y)^2-(x-y)^2\right]\\
&=\frac{1}{4}\left[\sigma_2(x+y)+\sigma_2(-x-y)-\sigma_2(x-y)-\sigma_2(-x+y)\right],
\end{split}
\end{equation*}
we conclude that $D_iu_q^{(2)}$ can be implemented by a $\mathrm{ReLU}$-$\relu2$ network. We have $$\mathcal{D}\left(\sigma_1\left(\sum_{j=1}^{n_1}a_{qj}^{(2)}u_j^{(1)}+b_q^{(2)}\right)\right)=3, \mathcal{W}\left(\sigma_1\left(\sum_{j=1}^{n_1}a_{qj}^{(2)}u_j^{(1)}+b_q^{(2)}\right)\right)\leq\mathcal{W}$$ and $$\mathcal{D}\left(\sum_{j=1}^{n_1}a_{qj}^{(2)}D_iu_j^{(1)}\right)=2, \mathcal{W}\left(\sum_{j=1}^{n_1}a_{qj}^{(2)}D_iu_j^{(1)}\right)\leq\mathcal{W}.$$ Thus $\mathcal{D}\left(D_iu_q^{(2)}\right)=4,$ $\mathcal{W}\left(D_iu_q^{(2)}\right)\leq\max\{2\mathcal{W},4\}$.

Now we apply induction for layers $k\geq3$. For the third layer,
\begin{equation*}
D_iu_q^{(3)}=D_i\sigma_2\left(\sum_{j=1}^{n_2}a_{qj}^{(3)}u_j^{(2)}+b_q^{(3)}\right)
=2\sigma_1\left(\sum_{j=1}^{n_2}a_{qj}^{(3)}u_j^{(2)}+b_q^{(3)}\right)\cdot\sum_{j=1}^{n_2}a_{qj}^{(3)}D_iu_j^{(2)}.
\end{equation*}
Since
$\mathcal{D}\left(\sigma_1\left(\sum_{j=1}^{n_2}a_{qj}^{(3)}u_j^{(2)}+b_q^{(3)}\right)\right)=4$, $\mathcal{W}\left(\sigma_1\left(\sum_{j=1}^{n_2}a_{qj}^{(3)}u_j^{(2)}+b_q^{(3)}\right)\right)\leq\mathcal{W}$ and $$\mathcal{D}\left(\sum_{j=1}^{n_2}a_{qj}^{(3)}D_iu_j^{(2)}\right)=4, \mathcal{W}\left(\sum_{j=1}^{n_1}a_{qj}^{(3)}D_iu_j^{(2)}\right)\leq\max\{2\mathcal{W},4\mathcal{W}\}=4\mathcal{W},$$ we conclude that $D_iu_q^{(3)}$ can be implemented by a $\mathrm{ReLU}$-$\relu2$ network and $\mathcal{D}\left(D_iu_q^{(3)}\right)=5$, $\mathcal{W}\left(D_iu_q^{(3)}\right)\leq\max\{5\mathcal{W},4\}=5\mathcal{W}$.

We assume that $D_iu_q^{(k)}(q=1,2,\cdots,n_k)$ can be implemented by a $\mathrm{ReLU}$-$\relu2$ network and $\mathcal{D}\left(D_iu_q^{(k)}\right)=k+2$, $\mathcal{W}\left(D_iu_q^{(3)}\right)\leq(k+2)\mathcal{W}$. For the $(k+1)-$th layer,
\begin{align*}
&D_iu_q^{(k+1)}=D_i\sigma_2\left(\sum_{j=1}^{n_k}a_{qj}^{(k+1)}u_j^{(k)}+b_q^{(k+1)}\right)\\
&=2\sigma_1\left(\sum_{j=1}^{n_k}a_{qj}^{(k+1)}u_j^{(k)}+b_q^{(k+1)}\right)\cdot\sum_{j=1}^{n_k}a_{qj}^{(k+1)}D_iu_j^{(k)}.
\end{align*}
Since
$\mathcal{D}\left(\sigma_1\left(\sum_{j=1}^{n_k}a_{qj}^{(k+1)}u_j^{(k)}+b_q^{(k+1)}\right)\right)=k+2$, $\mathcal{W}\left(\sigma_1\left(\sum_{j=1}^{n_k}a_{qj}^{(k+1)}u_j^{(k)}+b_q^{(k+1)}\right)\right)\leq\mathcal{W}$ and $\mathcal{D}\left(\sum_{j=1}^{n_k}a_{qj}^{(k+1)}D_iu_j^{(k)}\right)=k+2$, $\mathcal{W}\left(\sum_{j=1}^{n_k}a_{qj}^{(k+1)}D_iu_j^{(k)}\right)\leq\max\{(k+2)\mathcal{W},4\mathcal{W}\}=(k+2)\mathcal{W}$, we conclude that $D_iu_q^{(k+1)}$ can be implemented by a $\mathrm{ReLU}$-$\relu2$ network and $\mathcal{D}\left(D_iu_q^{(k+1)}\right)=k+3$, $\mathcal{W}\left(D_iu_q^{(k+1)}\right)\leq\max\{(k+3)\mathcal{W},4\}=(k+3)\mathcal{W}$.

Hence we derive that $D_iu=D_iu_1^{\mathcal{D}}$ can be implemented by a $\mathrm{ReLU}$-$\relu2$ network and $\mathcal{D}\left(D_iu\right)=\mathcal{D}+2$, $\mathcal{W}\left(D_iu\right)\leq\left(\mathcal{D}+2\right)\mathcal{W}$. Finally we obtain that $\mathcal{D}\left(\|\nabla u\|^2\right)=\mathcal{D}+3$, $\mathcal{W}\left(\|\nabla u\|^2\right)\leq d\left(\mathcal{D}+2\right)\mathcal{W}$.

The proof of (\ref{rg2c}) follows from Lemma \ref{Redm2}.
\end{proof}

{ By Lemma \ref{sym}, Lemma \ref{Redm1}, Lemma \ref{Redm2} and Lemma \ref{sym2}, we have to find upper bounds for  the
converging numbers $\mathcal{C}\left( \epsilon, \mathcal{N}^{2}, n\right)$  and $\mathcal{C}\left( \epsilon, \mathcal{N}^{1,2}, n\right)$ used in the Dudley's entropy formula. To this end, we need the VC-dimension   \cite{vapnik2015uniform} and Pseudo-dimension  \cite{anthony2009neural}.}

%We now introduce the concept of VC-dimension \cite{vapnik2015uniform}, which plays an important role in the research of neural network approximation. Let X be a space.

%\begin{Definition}
%	The growth function of $H$ is defined by
%	$$\mathcal{G}_{H}(m):=\max \left\{\left|H_{S}\right|: S \subset X,|S|=m\right\}, \quad \text { for } m \in \mathbb{N}$$	
%\end{Definition}

\begin{Definition}
Let $\mathcal{N}$ be a set of functions from $X = \Omega (\partial \Omega)$ to $\{0,1\}$. Suppose that $S=\{x_1,x_2,\cdots,x_n\}\subset X$. We say that $S$ is shattered by $\mathcal{N}$ if for any $b\in\{0,1\}^n$, there exists a $u\in\mathcal{N}$ satisfying
\begin{equation*}
u(x_i)=b_i,\quad i=1,2,\dots,n
\end{equation*}
\end{Definition}

\begin{Definition}
The VC-dimension of $\mathcal{N}$, denoted as $\mathrm{VCdim}(\mathcal{N})$, is defined to be the maximum cardinality among all sets shattered by $\mathcal{N}$.
\end{Definition}

VC-dimension reflects the capability of a class of functions to perform binary classification of
points. The larger VC-dimension is, the stronger the capability to perform binary classification is.  For more discussion of VC-dimension, readers are referred to \cite{anthony2009neural}.

For real-valued functions, we can generalize the concept of VC-dimension into pseudo-dimension \cite{anthony2009neural}.
\begin{Definition}
Let $\mathcal{N}$ be a set of functions from $X$ to $\mathbb{R}$. Suppose that $S=\{x_1,x_2,\cdots,x_n\}\subset X$. We say that $S$ is pseudo-shattered by $\mathcal{N}$ if there exists $y_1,y_2,\cdots,y_n$ such that for any $b\in\{0,1\}^n$, there exists a $u\in\mathcal{N}$ satisfying
\begin{equation*}
\mathrm{sign}(u(x_i)-y_i)=b_i,\quad i=1,2,\dots,n
\end{equation*}
and we say that $\{y_i\}_{i=1}^{n}$ witnesses the shattering.
\end{Definition}

\begin{Definition}
The pseudo-dimension of $\mathcal{N}$, denoted as $\mathrm{Pdim}(\mathcal{N})$, is defined to be the maximum cardinality among all sets pseudo-shattered by $\mathcal{N}$.
\end{Definition}

The following proposition showing a relation between uniform covering number and pseudo-dimension. % appears in \cite{anthony2009neural}.
\begin{Proposition}[Theorem 12.2, \cite{anthony2009neural}]\label{covering number pdim}
Let $\mathcal{N}$ be a set of real functions from a domain $X$ to the bounded interval $[0, \mathcal{B}]$. Let $\epsilon>0$. Then
\begin{equation*}
\mathcal{C}_{\infty}(\epsilon, \mathcal{N}, n) \leq \sum_{i=1}^{\mathrm{Pdim}(\mathcal{N})}\left(\begin{array}{c}
n \\
i
\end{array}\right)\left(\frac{\mathcal{B}}{\epsilon}\right)^{i},
\end{equation*}
which is less than $\left(\frac{en\mathcal{B}}{\epsilon\cdot\mathrm{Pdim}(\mathcal{N})}\right)^{\mathrm{Pdim}(\mathcal{N})}$ for $n\geq\mathrm{Pdim}(\mathcal{N})$.
\end{Proposition}

We now present the bound of pseudo-dimension for the $\mathcal{N}^{1,2}$, the class of  network functions with $\mathrm{ReLU}$ and $\mathrm{ReLU}^2$ activation functions. We first need a lemma stated below.
\begin{Lemma} \label{lemma covering number}
Let $p_1,\cdots,p_m$ be polynomials with $n$ variables  of degree at most $d$. If $n\leq m$, then
\begin{equation*}
|\{(\mathrm{sign}(p_1(x)),\cdots,\mathrm{sign}(p_m(x))):x\in\mathbb{R}^n\}|\leq2\left(\frac{2emd}{n}\right)^n
\end{equation*}
\end{Lemma}
\begin{proof}
See Theorem 8.3 in \cite{anthony2009neural}.
\end{proof}

\begin{Theorem}\label{pdimb}
Let
\begin{equation*}
\begin{split}
\mathcal{N}:=\{&u\in[0,1]^d: u \text{ can be implemented by a neural network}\\
&\text{with depth no more than }\mathcal{D} \text{ and width no more than }\mathcal{W},\\
&\text{and activation function in each unit  be the ReLU  or the}  \ \  \mathrm{ReLU}^2.\}
\end{split}
\end{equation*}	
Then
\begin{equation*}
\mathrm{Pdim}(\mathcal{N})=\mathcal{O}(\mathcal{D}^2\mathcal{W}^2(\mathcal{D}+\log\mathcal{W})).
\end{equation*}	
\end{Theorem}
\begin{proof}
The argument is follows from the  proof of Theorem 6 in \cite{bartlett2019nearly}.  The result stated here is somewhat stronger then Theorem 6 in \cite{bartlett2019nearly} since $\mathrm{VCdim}(\mathrm{sign}(\mathcal{N}))\leq \mathrm{Pdim}(\mathcal{N})$.

We consider a new set of functions:
\begin{equation*}
\mathcal{\widetilde{N}}=\{\widetilde{u}(x,y)=\mathrm{sign}(u(x)-y): u\in\mathcal{H}\}
\end{equation*}
It is clear that $\mathrm{Pdim}(\mathcal{N})\leq\mathrm{VCdim}(\mathcal{\widetilde{N}})$. We now bound the VC-dimension of $\mathcal{\widetilde{N}}$. Denoting $\mathcal{M}$ as the total number of parameters(weights and biases) in the neural network implementing functions in $\mathcal{N}$, in our case we want to derive the uniform bound for
\begin{equation*}
K_{\{x_i\},\{y_i\}}(m):=|\{(\operatorname{sign}(f(x_{1}, a)-y_1), \ldots, \operatorname{sign}(u(x_{m}, a)-y_m)): a \in \mathbb{R}^{\mathcal{M}}\}|
\end{equation*}
over all $\{x_i\}_{i=1}^{m}\subset X$ and $\{y_i\}_{i=1}^{m}\subset\mathbb{R}$. Actually the maximum of $K_{\{x_i\},\{y_i\}}(m)$ over all $\{x_i\}_{i=1}^{m}\subset X$ and $\{y_i\}_{i=1}^{m}\subset\mathbb{R}$ is the growth function $\mathcal{G}_{\mathcal{\widetilde{N}}}(m)$.
 In order to apply Lemma $\ref{lemma covering number}$, we partition the parameter space $\mathbb{R}^{\mathcal{M}}$ into several subsets to ensure that in each subset $u(x_i,a)-y_i$ is a polynomial with respcet to $a$ without any breakpoints. In fact, our partition is exactly the same as the partition in \cite{bartlett2019nearly}. Denote the partition as $\{P_1,P_2,\cdots,P_N\}$ with some integer $N$ satisfying
\begin{equation} \label{pdimb1}
N\leq\prod_{i=1}^{\mathcal{D}-1}2\left(\frac{2emk_i(1+(i-1)2^{i-1})}{\mathcal{M}_i}\right)^{\mathcal{M}_i}
\end{equation}
where $k_i$ and $\mathcal{M}_i$ denotes the number of units at the $i$th layer and the total number of parameters at the inputs to units in all the layers up to layer $i$ of the neural network implementing functions in $\mathcal{N}$, respectively. See \cite{bartlett2019nearly} for the construction of the partition. Obviously we have
\begin{equation} \label{pdimb2}
K_{\{x_i\},\{y_i\}}(m)\leq\sum_{i=1}^{N}|\{(\operatorname{sign}(u(x_{1}, a)-y_1), \cdots, \operatorname{sign}(u(x_{m}, a)-y_m)): a\in P_i\}|
\end{equation}
Note that $u(x_i,a)-y_i$ is a polynomial with respect to $a$ with degree the same as the degree of $u(x_i,a)$, which is equal to $1 + (\mathcal{D}-1)2^{\mathcal{D}-1}$ as shown in \cite{bartlett2019nearly}.  Hence by Lemma $\ref{lemma covering number}$, we have
\begin{align}
&|\{(\operatorname{sign}(u(x_{1}, a)-y_1), \cdots, \operatorname{sign}(u(x_{m}, a)-y_m)): a\in P_i\}|\nonumber\\
&\leq2\left(\frac{2em(1+(\mathcal{D}-1)2^{\mathcal{D}-1})}{\mathcal{M}_{\mathcal{D}}}\right)^{\mathcal{M}_{\mathcal{D}}}\label{pdimb3}.
\end{align}
Combining $(\ref{pdimb1}), (\ref{pdimb2}), (\ref{pdimb3})$ yields
\begin{equation*}
K_{\{x_i\},\{y_i\}}(m)\leq\prod_{i=1}^{\mathcal{D}}2\left(\frac{2emk_i(1+(i-1)2^{i-1})}{\mathcal{M}_i}\right)^{\mathcal{M}_i}.
\end{equation*}
We then have
\begin{equation*}
\mathcal{G}_{\mathcal{\widetilde{N}}}(m)\leq\prod_{i=1}^{\mathcal{D}}2\left(\frac{2emk_i(1+(i-1)2^{i-1})}{\mathcal{M}_i}\right)^{\mathcal{M}_i},
\end{equation*}
since the maximum of $K_{\{x_i\},\{y_i\}}(m)$ over all $\{x_i\}_{i=1}^{m}\subset X$ and $\{y_i\}_{i=1}^{m}\subset\mathbb{R}$ is the growth function $\mathcal{G}_{\mathcal{\widetilde{N}}}(m)$. Some algebras  as that of  the proof of Theorem 6 in \cite{bartlett2019nearly}, we obtain
\begin{equation*}
\mathrm{Pdim}(\mathcal{N})\leq\mathcal{O}\left(\mathcal{D}^2\mathcal{W}^2 \log \mathcal{U}+\mathcal{D}^3\mathcal{W}^2\right)
=\mathcal{O}\left(\mathcal{D}^2\mathcal{W}^2\left( \mathcal{D}+\log \mathcal{W}\right)\right)
\end{equation*}
where $\mathcal{U}$ refers to the number of units of the neural network implementing functions in $\mathcal{N}$.
\end{proof}

\subsection{Main results}
With the above preparation we present the main results of this paper in the scenario  $\mathcal{E}_{opt} =0$.
\begin{Theorem}\label{errmain}
Let $u^*$ is the solution of (\ref{mse}) with bounded  $f,g,w$.
$\widehat{u}_{\phi}$ is the minimizer of deep Ritz method defined in (\ref{objsam}) with $n = N(M)$ random samples.  If we set  the network parameters depth and width
 in the $\mathrm{ReLU}^2$ network $\mathcal{N}^2_{\mathcal{D},\mathcal{W},\mathcal{B}}$ as  $$\mathcal{D} \leq \lceil\log_2d\rceil+3, \mathcal{W}\leq \mathcal{O}(4d\left \lceil   n^{\frac{1}{d+2+\nu}}-4\right \rceil^d).$$ Let
 $\mathcal{B}$ be    the constraint on the output of function value and gradient norm of  $u\in \mathcal{N}^2$ according to (\ref{boundb}).
Then,
 $$\mathbb{E}_{\mathbf{X},\mathbf{Y}}[\|\widehat{u}_{\phi}-u^{*}\|_{H^{1}(\Omega)}^2]\leq C_{\mathcal{B},c_1,c_2,c_3,d} {\mathcal{O}} (N^{-1/(d+2+\nu)}+M^{-1/(d+2+\nu)}),$$
 where $\nu>0$ but can be arbitrary small.
\end{Theorem}
\begin{proof}
In order to apply Lemma \ref{Redm2}, we need to bound  the term
\begin{equation*}
\begin{split}
&\frac{1}{\sqrt{n}}\int_{\delta}^{\mathcal{B}}\sqrt{\log(2\mathcal{C}(\epsilon,\mathcal{N},n))}d\epsilon \\
&\leq \frac{\mathcal{B}}{\sqrt{n}}+ \frac{1}{\sqrt{n}}\int_{\delta}^{B}\sqrt{\log\left(\frac{en\mathcal{B}}{\epsilon\cdot\mathrm{Pdim}(\mathcal{N})}\right)^{\mathrm{Pdim}(\mathcal{N})}}d\epsilon\\
&\leq \frac{\mathcal{B}}{\sqrt{n}}+ \left(\frac{\mathrm{Pdim}(\mathcal{N})}{n}\right)^{1/2}\int_{\delta}^{B}\sqrt{\log\left(\frac{en\mathcal{B}}{\epsilon\cdot\mathrm{Pdim}(\mathcal{N})}\right)}d\epsilon,
\end{split}
\end{equation*}
where in the first inequality  we use  Proposition $\ref{covering number pdim}$. Now we calculate the integral. Set
\begin{equation*}
t=\sqrt{\log\left(\frac{en\mathcal{B}}{\epsilon\cdot\mathrm{Pdim}(\mathcal{N})}\right)}
\end{equation*}
then $\epsilon=\frac{en\mathcal{B}}{\mathrm{Pdim}(\mathcal{N})}\cdot e^{-t^2}$. Denote $t_1=\sqrt{\log\left(\frac{en\mathcal{B}}{\mathcal{B}\cdot\mathrm{Pdim}(\mathcal{N})}\right)}$, $t_2=\sqrt{\log\left(\frac{en\mathcal{B}}{\delta\cdot\mathrm{Pdim}(\mathcal{N})}\right)}$.  And
\begin{equation*}
\begin{split}
&\int_{\delta}^{\mathcal{B}}\sqrt{\log\left(\frac{en\mathcal{B}}{\epsilon\cdot\mathrm{Pdim}(\mathcal{N})}\right)}d\epsilon\\
&=\frac{2en\mathcal{B}}{\mathrm{Pdim}(\mathcal{N})}\int_{t_1}^{t_2}t^2e^{-t^2}dt\\
&=\frac{2en\mathcal{B}}{\mathrm{Pdim}(\mathcal{N})}\int_{t_1}^{t_2}t\left(\frac{-e^{-t^2}}{2}\right)'dt\\
&=\frac{en\mathcal{B}}{\mathrm{Pdim}(\mathcal{N})}\left[t_1e^{-t_1^2}-t_2e^{-t_2^2}+\int_{t_1}^{t_2}e^{-t^2}dt\right]\\
&\leq\frac{en\mathcal{B}}{\mathrm{Pdim}(\mathcal{N})}\left[t_1e^{-t_1^2}-t_2e^{-t_2^2}+(t_2-t_1)e^{-t_1^2}\right]\\
&\leq\frac{en\mathcal{B}}{\mathrm{Pdim}(\mathcal{N})}\cdot t_2e^{-t_1^2}=\mathcal{B}\sqrt{\log\left(\frac{en\mathcal{B}}{\delta\cdot\mathrm{Pdim}(\mathcal{N})}\right)}.
\end{split}
\end{equation*}
Choosing $\delta=\mathcal{B}\left(\frac{\mathrm{Pdim}(\mathcal{N})}{n}\right)^{1/2}\leq\mathcal{B}$, by Lemma \ref{Redm2} and the above display, we get for both  $$\mathcal{N} = \mathcal{N}^2, \\ \mathrm{and}\ \   \mathcal{N} = \mathcal{N}^{1,2}$$ there holds
\begin{align}
&\mathfrak{R}(\mathcal{N}) \nonumber \\
& \leq 4\delta+\frac{12}{\sqrt{n}}\int_{\delta}^{\mathcal{B}}\sqrt{\log(2\mathcal{C}(\epsilon,\mathcal{N},n))}d\epsilon \nonumber \\
&\leq4\delta+  \frac{12\mathcal{B}}{\sqrt{n}}+ 12\mathcal{B}\left(\frac{\mathrm{Pdim}(\mathcal{N})}{n}\right)^{1/2}\sqrt{\log\left(\frac{en\mathcal{B}}{\delta\cdot\mathrm{Pdim}(\mathcal{N})}\right)} \nonumber \\
&\leq28\sqrt{\frac{3}{2}}\mathcal{B}\left(\frac{\mathrm{Pdim}(\mathcal{N})}{n}\right)^{1/2}\sqrt{\log\left(\frac{en}{\mathrm{Pdim}(\mathcal{N})}\right)} \label{emperr}.
\end{align}
Then by Lemma \ref{errstadec}, \ref{Redm1}, \ref{sym}, \ref{Redm2}, \ref{sym2} and equation  (\ref{emperr}), we have
\begin{align*}
&{\mathcal{E}_{sta}}  = 2\sup_{u \in \mathcal{N}^2} |\mathcal{L}(u) - \widehat{\mathcal{L}}(u) |\\
&\leq 2 (c_3+c_3+c_3^2)\mathfrak{R}(\mathcal{N}^2) + 2 \mathfrak{R}(\mathcal{N}^{1,2})\\
&\leq 56\sqrt{\frac{3}{2}}(2c_3+c_3^2)\mathcal{B}\left(\frac{\mathrm{Pdim}(\mathcal{N}^2)}{n}\right)^{1/2}\sqrt{\log\left(\frac{en}{\mathrm{Pdim}(\mathcal{N}^2)}\right)}\\
&+56\sqrt{\frac{3}{2}}\mathcal{B}\left(\frac{\mathrm{Pdim}(\mathcal{N}^{1,2})}{n}\right)^{1/2}\sqrt{\log\left(\frac{en}{\mathrm{Pdim}(\mathcal{N}^{1,2})}\right)}.
\end{align*}
Since $\mathrm{Pdim}(\mathcal{N}^{2}) \leq \mathrm{Pdim}(\mathcal{N}^{1,2})$,
plugging the upper bound of $\mathrm{Pdim}(\mathcal{N}^{1,2})$ derived in Theorem \ref{pdimb} into the above display  and using the relationship of depth and width between $\mathcal{N}^2$ and $\mathcal{N}^{1,2}$ proved in Lemma \ref{sym2}, we get
\begin{equation}\label{sterrf}
\mathcal{E}_{sta} \leq C_{\mathcal{B},c_3}\left[d(\mathcal{D}+3)(\mathcal{D}+2)\mathcal{W}\sqrt{\frac{\mathcal{D}+3+\log (d(\mathcal{D}+2)\mathcal{W})}{n}}\right]^{1-\nu},
\end{equation}
where $\nu>0$ can be arbitrarily small. Combing (\ref{sterrf}) with the approximation error in Theorem \ref{errapp} and  taking $\epsilon^{2} = C_{\mathcal{B},c_1,c_2,c_3}C\left(\frac{1}{n}\right)^{\frac{1}{d+2+\nu}}$, we get that
\begin{align*}
&\mathbb{E}_{\mathbf{X},\mathbf{Y}}[\|\widehat{u}_{\phi}-u^{*}\|_{H^{1}(\Omega)}^2]\\
&\leq  \frac{2}{c_1\wedge 1} \left[C_{\mathcal{B},c_3}\left( d(\mathcal{D}+3)(\mathcal{D}+2)\mathcal{W}\sqrt{\frac{\mathcal{D}+3+\log (d(\mathcal{D}+2)\mathcal{W})}{n}}\right)^{1-\mu} + \frac{c_3+1}{2}\epsilon^2\right]\\
&\leq \frac{2}{c_1\wedge 1} \left[C_{\mathcal{B},c_3}\left( 4d^2(\lceil\log d\rceil+6)(\lceil\log d\rceil+5)\left \lceil   \frac{Cc_2}{\epsilon}-4\right \rceil^d\cdot\right.\right. \\
&\quad\left.\left.\sqrt{\frac{\lceil\log d\rceil+6+\log \left(d(\lceil\log d\rceil+5)\cdot4d\left \lceil   \frac{Cc_2}{\epsilon}-4\right \rceil^d\right)}{n}}\right)^{1-\mu} + \frac{c_3+1}{2}\epsilon^2\right]\\
&\leq C_{\mathcal{B},c_1,c_2,c_3,d} {\mathcal{O}} \left(n^{-1/(d+2+\nu)}\right).
\end{align*}
%where $\nu>0$ can be arbitrarily small.
\end{proof}

\begin{Remark}
Deep Ritz method  is actually a kinds of deep nonparametric estimation method   where we estimate functions from random samples.
The benefit is that we can use DRM to handle PDEs in high dimension since  only  a small bach of samples
is needed during  SGD training. In contrast, we have form the loading matrix and vector explicitly in FEM.
However, what we have to pay is that  the convergence rate as illustrate as follows.
Let $N= N=\mathcal{O}(\frac{1}{h^d})$, where $h$ is the size of the mesh in FEM. From Theorem \ref{errmain}, we get $$\mathbb{E}_{\mathbf{X},\mathbf{Y}}[\|\widehat{u}_{\phi}-u^{*}\|_{H^{1}(\Omega)}^2]=\mathcal{O}(h^{\frac{d}{d+2+\nu}}).$$
In the case $d=2$,
by Markov's inequality and the above display, we get   with  high probability, $$\|\widehat{u}_{\phi}-u^{*}\|_{H^{1}(\Omega)}\leq \mathcal{O}(h^{\frac{1}{4+\nu}}).$$ Comparing the well known results of FEM where the convergence rate in $H^1$ norm is $\mathcal{O}(h)$, the rate proved here for the DRM method is far from satisfactory.

In the literature of nonparametric regression, where functions  living in certain function classes are  estimated from  $n$ paired random samples,
the best convergence rate   in $H^1$ norm for estimating functions in $H^2$  is  $\mathcal{O}(n^{-\frac{1}{4+d}})$ \cite{stone1982optimal}.
Obviously, the nonparametric learning  task  in DRM is not easier than nonparametric regressions.
What we proved here is $\mathcal{O}(n^{-\frac{1}{4+2d+\nu}})$ in $H^1$ norm with $\nu$ can be arbitrary small.
\textbf{We conjecture that  the best convergence rate in $H^1$ norm of DRM with both $n$ training samples in both the domain and boundary is also  $\mathcal{O}(n^{-\frac{1}{4+d}})$.}   Upon this conjecture,  we get the optimal convergence rate of DRM  for  $d=2$ is  $\|\widehat{u}_{\phi}-u^{*}\|_{H^{1}(\Omega)}\leq \mathcal{O}(h^{\frac{1}{3}})$ with high probability.

The convergence rate of DRM  proved here suffers the curse of dimensionality.  One possible direction to improve  the convergence rate  and  reduce the curse of dimensionality is considering    solutions of PDE (\ref{mse}) with higher  regularity. For example,  if we  assume  $f\in H^s(\Omega),  s \geq 1$,   deep Ritz method using deep neural network with $\mathrm{ReLU}^{s+2}$ activation functions  will reduce the curse since the higher regularity  assumption  can improve both the approximation and statistical error.  We leave the detail of this idea in  a following up work.
\end{Remark}

\section{Conclusion and extension}
In this paper, we provide a rigorous numerical analysis on deep Ritz method  (DRM) \cite{wan11} for second order elliptic equations with Neumann boundary conditions.
We establish the first   nonasymptotic convergence rate in $H^1$ norm  on   DRM for general deep networks with $\mathrm{ReLU}^2$ activation functions.
In addition to provide theoretical justification of DRM, our study  also provides   guidance   on how to set the  hyper-parameter  of depth and width to achieve the desired convergence rate in terms of number of training samples.
Technically, we derive   bounds on  the approximation error of deep $\mathrm{ReLU}^2$ network in $H^1$ norm and on the Rademacher complexity  of the non-Lipschitz composition of  gradient norm  and   $\mathrm{ReLU}^2$ network.

There are several directions for our future exploration. First, it is easy to extend the current analysis  for general second order elliptic equations with variational form under  Dirichlet or Robin boundary conditions. Second, the approximation  and statistical error  bounds deriving here can be used for  studying  the nonasymptotic convergence rate for  residual based method (PINNS).
Studying deep DGM by combing current analysis with  the tools for analyzing GAN \cite{goodfellow14} is also of immense interest.

\section*{Acknowledgements}
 Y.  Jiao was supported in part by
the National Science Foundation of China
under Grant 11871474  and by the research fund
of KLATASDSMOE.
X. Lu is partially supported by the National Science Foundation of China (No. 11871385), the National Key Research and Development Program of China (No.2018YFC1314600) and the Natural Science Foundation of Hubei Province (No. 2019CFA007).
J.  Yang is supported by National Science Foundation of China (No. 12071362 and 11671312), the National Key Research and Development Program of China (No. 2020YFA0714200), the Natural Science Foundation of Hubei Province (No. 2019CFA007).
The numerical simulations in this work have been done on the supercomputing system in the Supercomputing Center of Wuhan University.

\bibliographystyle{siam}
\bibliography{ref}	
\end{document}